\documentclass[11pt]{article}
\usepackage{amsmath,amsthm,amssymb,mathrsfs,amscd,amsthm,amscd,amsfonts,graphicx}
\usepackage[cmtip,all]{xy}

\usepackage{eucal}
\usepackage[T1]{fontenc}
\usepackage[utf8]{inputenc}
\usepackage{indentfirst,url,xypic}
\usepackage{lipsum}
\usepackage{faktor}
\usepackage[all]{xy}
\usepackage{url}
\usepackage[colorlinks,plainpages]{hyperref}
\usepackage{verbatim}
\usepackage{fancyvrb}
\usepackage{listings}
\usepackage{mathtools}
\usepackage{algorithm}
\usepackage[noend]{algorithmic}

\usepackage[colorlinks,plainpages]{hyperref}
\hypersetup{
	colorlinks=true,
	linkcolor=blue,
	filecolor=magenta,      
	urlcolor=cyan,
	citecolor=blue
}

\newtheorem{Definition}{ Definition}[section]
\newtheorem{Theorem}[Definition]{Theorem}
\newtheorem{Remark}[Definition]{Remark}
\newtheorem{Proposition}[Definition]{Proposition }
\newtheorem{Corollary}[Definition]{Corollary}

\newtheorem{Proof}[Definition]{Proof}
\newtheorem{Proof*}{Proof}
\newtheorem{Lemma}[Definition]{Lemma}

\newcommand{\Z}{{\mathbb Z}}

\newcommand{\Q}{{\mathbb Q}}

\newcommand{\F}{{\mathbb F}}

\newcommand{\fp}{\mathfrak{p}}
\newcommand{\fq}{\mathfrak{q}}

\DeclareMathAlphabet{\mathsc}{U}{rsfs}{m}{n}

\newcommand{\x}{{\langle x \rangle}}
\DeclareMathOperator{\Spec}{Spec}

\DeclareMathOperator{\Quot}{Quot}

\begin{document}

\title{Using Semicontinuity for Standard Bases Computations}
\author{Gert-Martin Greuel, Gerhard Pfister, Hans Schönemann}
%\email{greuel@mathematik.uni-kl.de}

\maketitle
%\date{\today}
\begin{center} Dedicated to the memory of Vladimir Gerdt
\end{center} 

\begin{abstract}
We present new results and an algorithm for standard basis computations of a 0-dimensional ideal $I$ in a  power series ring or in the localization of a polynomial ring in finitely many variables over a field K.
The algorithm provides a significant speed up if $K$ is the quotient field of a Noetherian integral domain $A$, when coefficient swell occurs.
The most important special cases are perhaps $A=\Z$ resp $A=k[t]$, $k$ any field and $t$ a set of parameters. Given $I$ as an ideal in the polynomial ring over $A$, we compute first a standard basis modulo a prime number $p$, resp. specializing $t=p\in k^s$.
We then use the 
"highest corner" of the specialized ideal to cut off high order terms from the polynomials during the standard basis computation over $K$ to get the speed up. An important fact is that we can choose  $p$ as an arbitrary prime resp. as an arbitrary element of $k^s$, not just a "lucky" resp. "random" one. Correctness of the algorithm will be deduced from a general semicontinuity theorem in \cite{GP20}. The computer algebra system {\sc Singular} provides already the functionality to realize the algorithm  and we present several examples illustrating its power.\\

\end{abstract}

%\renewcommand{\contentsname}{Table of Contents}
%\tableofcontents

\addcontentsline{toc}{section}{Introduction}

%\vspace{1.0cm}

\noindent \textbf{\Large Introduction}\\

% When studying numerical invariants of singularities of an algebraic variety $X$ at a point $p\in X$, one is often lead at some point to compute the dimension  
% %$\dim_{K}K[x]_\x/I$ for some field $K$ and $I$ an ideal in $K[x]$. 
% $\dim_{k(p)}\ko_{X,p}/I(p)$, where $\ko_{X,p}$ is the local ring of $X$ at $p$, $I(p)$ an ideal in $\ko_{X,p}$ and $k(p)$ the residue field of $p$. By choosing an affine chart we may assume that 
% %$I$ is an ideal in $A[x]$,  $x=(x_1,...,x_n)$,  
% $\ko_{X,p} = k(p)[x]_\x$, $x=(x_1,...,x_n)$, and the computations usually require a standard basis of $I(p)$ with respect to a local monomial ordering on $k(p)[x]$ (cf. \cite{GP08}). 
 When studying numerical invariants of singularities of an algebraic variety $X$ at a point 
 one is often lead at some point to compute the dimension  $\dim_{K}K[x]_\x/I$ for some field $K$ and $I$ an ideal in $K[x]$. Typical examples are the Milnor number or the Tjurina number of an isolated hypersurfac singularity. 
% $\dim_{K}\ko_{X,x}/I$ for some field $K$, where $\ko_{X,x}$ is the local ring of $X$ at $x$, $I$ an ideal in $\ko_{X,x}$ and $k(x)$ the residue field of $x$. By choosing an affine chart we may assume $x=0$ and $\ko_{X,x} = k(x)[x]_\x, x=(x_1,...,x_n)$, and the computations usually require a standard basis of $I$ w.r.t. a local monomial ordering (cf. \cite{GP08}). 
 These computations can be very time and space consuming  for $K=\Q$ or $K=k(t)$, the quotient field of $k[t]$ for some field $k$ and finitely many parameters $t=(t_1,...,t_s)$ due to intermediate coefficient growth.

  Algebraic geometers usually choose then a  prime number $p\in \Z$ resp. substitute $t$ by an element $p \in k^s$, with $p$ "sufficiently general", and do the standard computation of the induced ideal $I(p)$ in $\F_p[x]_\x$ resp. $k[x]_\x$.
The computed dimension $\dim_{\F_p} \F_p[x]_\x/I(p)$ resp. $\dim_{k}k[x]_\x/I(p)$ coincides with high probability with $\dim_K K[x]_\x/I$. But to be sure, a standard basis computation over the field $K$ can in general not be avoided.
 
In the present paper we provide an algorithm for the computation of $\dim_K K[x]_\x/I$, where $K= \Quot(A)$ is the quotient field of a Noetherian integral domain $A$, which is significantly faster for non-trivial examples than previously known methods. Let us explain the main ideas for the special cases $A=\Z$, $K=\Q$ and $A=k[t]$, $K=k(t)$.  To simplify notations we set
$$ R(0):=K[x]_\x \text{ and } R(p):=k(p)[x]_\x,$$
where,  for $A=\Z$, $p$ is a prime number and $k(p):=\F_p$, while for   $A=k[t]$,  $p$ is an element of $k^s$ and $k(p):=k$. Let $I(p)\subset R(p)$ denote the ideal generated by the image of $I$ under the natural map $A[x] \to R(p)$.

We start the algorithm by  choosing an arbitrary $p$ and compute a standard basis  of the ideal $I(p)$ in $R(p)$ over $k(p)$. We assume that 
$$\dim_{k(p)} R(p)/I(p) <\infty,$$ 
 the dimension being computed by using the standard basis of $I(p)$. Since the codimension of $I(p)$ in $R(p)$ is finite, $I(p)$ has a so called "highest corner" $HC(I(p))$, introduced in \cite{GP08}, which is the smallest monomial (w.r.t. to the local monomial order) not contained in the leading ideal $L(I(p))$.
 $HC(I(p))$ can be read from the leading terms of the standard basis of $I(p)$. 
 We then compute a standard basis of the ideal $I(0):=IR(0)$ in $R(0)$ over $K$, but during the computation we
cut off high order terms from the polynomials, the bound for cutting off being deduced from $HC(I(p))$ (Corollary \ref{cor.trunc}). To prove correctness of the algorithm, i.e., that the result is in fact a standard basis of $I(0)$, we will prove that the highest corner is lower semicontinuous in the Zariski topology of $\Spec A$ (Proposition \ref{prop.scHC}). This is  in turn deduced from a general semicontinuity theorem in \cite{GP20} (cf. Theorem \ref{thm.sc}).
An important fact is that we can choose  $p$ as an arbitrary prime resp. as an arbitrary element of $k^s$, not just a  "sufficiently general" one.   

%The speed of the algorithm compared to computing a standard basis over $ K $ without truncating higher order terms is really amazing for non-trivial examples. 
We state the results and the algorithm for a general Noetherian domain $A$,
but illustrate this in Section \ref{sec.3} for the special cases $A=\Z$ and $k[t]$.
The impressive speed up for non-trivial examples is demonstrated at the end of the article.
The computations are done in the computer algebra system {\sc Singular}
(\cite{DGPS}), which provides the needed functionality: computations in $k[x]_\x$, commands for the highest corner and for truncation of polynomials.
\medskip

{\bf Notation:} The following notations are used throughout the paper.
$A$ denotes a Noetherian integral domain, $A[x]$, $x=(x_1,...,x_n)$, the 
polynomial ring over $A$ and $I\subset A[x]$ an ideal.
For a prime ideal $\fp$ in $A$ let $k(\fp):=A_\fp/\fp A_\fp = \Quot(A/\fp)$ be the residue field of the local ring $A_\fp$. In particular, $k(0):=k(\langle 0\rangle)=\Quot(A)$. The map $A \to k(\fp)$ induces natural maps 
$A[x] \to k(\fp)[x] \to k(\fp)[x]_\x \to (k(\fp)[x]_\x)^\wedge =  k(\fp)[[x]]$, with $^\wedge$ the $\x$-adic completion. We set 
$$ R(\fp):=k(\fp)[x]_\x \text{ and } \hat R(\fp):=k(\fp)[[x]],$$ 
and denote by  $I(\fp)$ resp. $\hat I(\fp)$ the ideal in  $R(\fp)$ resp. in $\hat R(\fp)$ generated by the image of $I$ under the above maps.
Instead of $R(\langle 0\rangle)$ and $I(\langle 0\rangle)$ we write $R(0)$ and $I(0)$. Note that 
$$\dim_{k(\fp)} R(\fp)/I(\fp)< \infty \iff \dim_{k(\fp)} \hat R(\fp)/\hat I(\fp) <\infty$$
and that then the dimensions coincide.

%%%%%%%%%%%%%%%%%%%%%%%%%%%%%%%%%%%%%

\section{Semicontinuity of the Highest Corner}\label{sec.1}

The following semicontinuity theorem is the main ingredient for the correctness of our algorithm.

\begin{Theorem} [\cite{GP20}, Theorem 42] \label{thm.sc} 
For a fixed $\fp \in \Spec A$ assume that $\dim_{k(\fp)} \hat R(\fp)/\hat I(\fp) <\infty$. Then there is an open neighbourhood $U$ of $\fp$  in $\Spec A$ such that  
$$\dim_{k(\fq)} \hat R(\fq)/\hat I(\fq) \leq  \dim_{k(\fp)} \hat R(\fp)/\hat I(\fp) \text { for all } \fq \in U.$$
\end{Theorem}

Since the prime ideal $\langle 0\rangle$ is contained in every neighbourhood of any prime ideal $\fp$, we get

\begin{Corollary}\label{cor.sc}
Let $\fp \in \Spec A$ be such that $\dim_{k(\fp)} R(\fp)/ I(\fp) <\infty$. Then 
$$\dim_{k(0)} R(0)/ I(0) \leq  \dim_{k(\fp)} R(\fp)/ I(\fp).$$
\end{Corollary}

 We recall now the definition of the highest corner from \cite[Definition 1.7.11]{GP08}. It uses the notion of a local monomial ordering $>$ on a polynomial ring
 $A[x_1,...,x_n]$, $A$ any ring, i.e. a local monomial ordering on
 $Mon(x_1, ... , x_n)$, the set of monomials (power products) in $x_1, ... , x_n$ ($1> x_i$ for each $i$) as defined in \cite[Section 1.2]{GP08}.
 
\begin{Definition} Let $K$ be a field, $I \subset K[x]$ an ideal, and $>$ a local monomial ordering on $K[x]$. A monomial $m$ is called the {\em highest corner} of $I$ (with respect to $>$), denoted by $HC(I)$, if
\begin{enumerate}
\item $m \notin L(I)$, where $L(I)\in K[x]$ is the leading ideal\,\footnote{The leading ideal of $I$ is the ideal generated by the leading monomials of all elements $\neq 0$ of $I$.} of $I$,
\item if $m'$ is a monomial with $m' <m$ $\Longrightarrow m' \in L(I)$.
\end{enumerate}
\end{Definition}

 It shown in \cite[Lemma 1.7.13]{GP08} that $HC(I)$ is the smallest monomial (with respect to $>$) not contained in $I$. 
\medskip

In the following we need local weighted degree orderings, where the variables have negative weights (see \cite[Definition 1.2.9]{GP08}). A local weighted degree ordering is called a {\em local degree ordering}
if weight$(x_i) = -1$ for each $i$.\footnote{Examples of local degree orderings in {\sc Singular} are {\tt ds} and  {\tt Ds}  and of local weighted degree orderings {\tt ws} and  {\tt Ws}.}

\begin{Lemma}\label{lem.hc}
Let  $\fp$ be a prime ideal in $A$ such that $\dim_{k(\fp)} R(\fp)/ I(\fp) < \infty$ and let $>$ be a local weighted degree ordering 
on $Mon(x_1, ... , x_n)$. Then $HC(I(\fp))$ and $HC(I(0))$ exist.
\end{Lemma}

 \begin{proof}
This follows from \cite[Lemma 1.7.14]{GP08} and Corollary \ref{cor.sc}.
\end{proof}

We prove now the semicontinuity of the highest corner.

\begin{Proposition}\label{prop.scHC}
With the assumptions of Lemma \ref{lem.hc}
let $M$ be a set of monomials $m$ such that $m< HC(I(\fp))$
and set  
$I'(0):= I(0)+\langle M\rangle_{R(0)}.$
\begin{enumerate}
\item [(i)] If $\dim_{k(\fp)} R(\fp)/ I(\fp) = \dim_{k(0)} R(0)/ I'(0)$
then $I'(0) = I(0)$.
\item [(ii)]  If $M$ is the set of all monomials $< HC(I(\fp))$ and $\dim_{k(\fp)} R(\fp)/ I(\fp) = \dim_{k(0)} R(0)/ I'(0)$, then
$$HC(I(0))\geq HC(I(\fp)).$$
\end{enumerate}
\end{Proposition}
 \begin{proof}
(i) From Corollary \ref{cor.sc} we get
$$\dim_{k(0)} R(0)/I'(0) \leq \dim_{k(0)} R(0)/I(0) \leq\dim_{k(\fp)} R(\fp)/I(\fp).$$ 
The assumption implies $\dim_{k(0)} R(0)/I'(0) = \dim_{k(0)} R(0)/I(0)$ and hence $I'(0) = I(0)$. 

(ii) Now let $M$ be the set of all monomials $< HC(I(\fp))$. Since $I'(0) = I(0)$, $m \in I(0)$ and hence $m \in L(I(0))$ for all $m < HC(I(\fp))$.
It follows that $HC(I(0)) \geq HC(I(\fp))$.
\end{proof}

\begin{Remark}{\em
Theorem \ref{thm.sc} is proved in \cite{GP20} more generally for submodules
 $I\subset A[x]^m$ with $\hat R(\fq)/\hat I(\fq)$ replaced by $\hat R(\fq)^m/\hat I(\fq)$, such that Corollary \ref{cor.sc} holds also for modules.
The highest corner can be defined in the same way for submodules of $K[x]^m$ and then Lemma \ref{lem.hc}, Proposition \ref{prop.scHC} and Algorithm \ref{alg} hold analogous for modules
(but we do not pursue this here).
}
\end{Remark}
%%%%%%%%%%%%%%%%%%%%%%%%%%%%%%%%%%

\section{A bound for Truncation}
\begin{Proposition} \label{prop.trunc}
Let $K$ be a field,  $>$ a local degree ordering on $K[x]$, 
$I \subset K[x]$ an ideal with $\dim_K K[x]_\x/I <\infty$,  and  $m\leq HC(I)$ a monomial. If $g_1,...,g_s$ a minimal standard basis\,\footnote{A standard basis of $I$ w.r.t. > is a set of elements $G=\{g_1,...,g_s\}$ $\subset$ $I$ s.t. $L(I)= \langle LM(g_1), ...,LM(g_s)\rangle_{K[x]}$. By \cite[Lemma 1.6.7]{GP08}  $G$ generates $IK[x]_\x$. $G$ is called {\em minimal} if no $g_i$ can be deleted, i.e., the leading term of $g_i$  is not divisible by $LM(g_j), j\neq i$. }
of $I$, then
$$ \deg(LM(g_i)) \leq \deg(m)+1  \text{ for all } i.\footnote{deg denotes the ususal degree, with all variables having degree 1.}$$ 
\end{Proposition}

\begin{proof} By definition of $HC(I)$, all monomials $ < HC(I)$ are contained in $L(I)$. Since we have a local degree ordering, all monomials of degree $\geq \deg(m)+1$ are contained in $L(I)$. If $\deg(LM(g_i)) > \deg(m)+1$ for some $i$, we can divide 
$LM(g_i)$ by some variable and get a monomial of degree $\geq \deg(m)+1$. This is contained in $L(I)$ and hence divisible by some $LM(g_j)$, contradicting the minimality of the standard basis.
\end{proof}

\begin{Corollary} \label{cor.trunc}
With the assumptions of Proposition \ref{prop.trunc} 
let $f_1,...,f_k$ be a set of generators of $I$. Set  $d:= \deg(m)+1$ and for a fixed monomial $m'$ with $\deg(m')  > d$
set $\hat f_i:= f_i + a_i m'$ with $a_i \in K$ arbitrary. Then $\hat f_1,...,\hat f_k$ generate $I K[x]_\x$. 

Moreover, if we omit any monomial of degree $>d$  from the involved polynomials during the standard basis computation of $I$,  the result is a standard basis of $I$.
\end{Corollary}

\begin{proof} We prove the second statement first. Let $M$ be the set of all monomials $m'$ of degree $=d+1$. Then $m'$  will not occur as a leading monomial of a minimal standard basis of $I$ by Proposition \ref{prop.trunc}.  Moreover, $m'\in I$ since $m' < HC(I)$,  and $I$ is generated by $G_0=\{f_1,...,f_k\}\cup M$.

 Now consider the algorithm for computing a reduced standard basis\,\footnote{A {\em reduced standard basis} of $I$ is a minimal standard basis $G=\{g_1,...,g_s\}$  such that for each $i$ the leading coefficient of $g_i$ is 1 and all monomials of $g_i$ different from its leading monomial are not in the leading ideal $L(I)$. For zero-dimensional ideals a reduced standard basis always exists since all monomials smaller than the highest corner are in the ideal.}
 from \cite[Algorithms 1.7.1 and 1.7.6] {GP08}. Starting with $G_0$ we build s-polynomials and (completely) reduce them by previously computed polynomials.  In the i-th step of the algorithm we get a set of generators $G_i$ of $I$ which finally will become a reduced standard basis of $I$. We have for $f,g \in G_i$
$$spoly(f,g) = \sum c_\alpha x^\alpha= \sum_{|\alpha| \leq d}c_\alpha x^\alpha + \sum_{|\alpha| > d}c_\alpha x^\alpha,$$
where the second summand will be reduced to $0$ when building the reduced normal form, due to the monomials in $M$. We see that the effect of the monomials in $M$ during reduction  is the same as omitting all monomials of degree $> d$. The algorithm stops with a reduced standard basis of $I$ and  thus the second part of the corollary follows.

We just proved that $f_1,...,f_k$ and $\hat f_1,...,\hat f_k$ lead to the same reduced standard basis. Since any standard basis generates the ideal $I K[x]_\x$ in $K[x]_\x$ by \cite[Lemma 1.6.7]{GP08}, the first  statement follows.
\end{proof}

\begin{Remark}{\em 
(1) For special degree orderings we can delete even more terms during the standard basis computation. 
Let  $>$  denote the 
negative degree reverse lexicographical ordering {\tt ds} (\cite[Example 1.2.8]{GP08}) and  $x_n<...<x_1 <1$. Then all monomials $m'<x_n m$ can be deleted.\\
(2) For a local weighted degree ordering let $d$ be the smallest integer such that $\x^d \subset I$. Then each monomial $m'$ in $\x^d$ satisfies $m' < HC(I)$ and the same arguments as above show that Proposition \ref{prop.trunc} and Corollary \ref{cor.trunc} hold with 
$\deg(m)+1$ replaced by $d$.
}
\end{Remark}
%%%%%%%%%%%%%%%%%%%%%%%%%%%%%%%%

\section{Special Cases and Algorithm} \label{sec.3}

We use the notations from the introduction, with $I\subset A[x]$ an ideal. 

\subsection{Special cases}
For the sake of convenience, we illustrate our results for the perhaps most important special cases $A=\Z$ and $A=k[t]$.
\medskip

$\bullet$ {\em $A=\Z$, computations over $\Q$ and $\F_p$}:\\
Let $p \in \Z$ be an {\em arbitrary} prime number.
For the prime ideals $\fp = \langle 0 \rangle$ resp.  $\fp = \langle p \rangle$, we have $k(\fp) = \Q$ resp. $k(\fp) = \F_p$. Then $R(0) = \Q[x]_\x$ resp. $R(p) = \F_p[x]_\x$. Let  $I(0)$  resp. $I(p)$ denote the induced ideal in  $R(0) $ resp. $R(p)$. If $I(p)$ is a 0-dimensional ideal we have by Corollary \ref{cor.sc}
$$\dim_{\Q}  \Q[x]_\x/ I(0) \leq  \dim_{\F_p} \F_p[x]_\x/ I(p).$$

$\bullet$ {\em $A=k[t]$, $t=(t_1,...,t_s)$, $k$ any field, computations over $k(t)$ and $k$}:\\
Let $p \in k^s$ be an {\em arbitrary} element.
For the prime ideals $\fp = \langle 0 \rangle$ resp.  $\fp = \langle t-p \rangle$, we have $k(\fp) = k(t)$ resp. $k(\fp) = k$.\footnote{\,We could choose any other maximal ideal $\fp$ of $k[t]$. Then $k(\fp)$ would be a finite field extension of $k$ and the algorithm works as well.} Let  $I(0) \subset R(0) = k(t)[x]_\x$ resp. $I(p) \subset R(p) = k[x]_\x$ be the induced ideals with $I(p)$ being 0-dimensional.
By Corollary \ref{cor.sc} we have
$$\dim_{k(t)}  k(t)[x]_\x/ I(0) \leq  \dim_{k} k[x]_\x/ I(p).$$
The $K$-dimensions in the above formulas do not change if we replace $K[x]_\x$ by $K[[x]]$ for the various fields $K$.

For $A=\Z[t]$ and $I\subset \Z[t][x]$ we can of course reduce modulo a prime number $p$ and substitute $t$ by $a\in \Z$ at the same time and get the ideal  $I(p,a)\subset \F_p[x]_\x$. We then have 
$\dim_{\Q(t)}  \Q(t)[x]_\x/ I(0) \leq  \dim_{\Q} \Q[x]_\x/ I(a)$ $\leq \dim_{\F_p} \F_p[x]_\x/ I(p,a)$ (Corollary \ref{cor.sc})
 and 
a standard basis computation of $I(p,a)$ is usually much faster than of $I(a) \subset \Q[x]_\x$, and $HC(I(p,a))$ can be used for deleting monomials.

\subsection {Algorithm}\label{alg}

Let $A$ be Noetherian domain of dimension $\geq 1$\footnote{If $\dim A = 0$, then $A$ is a field and the algorithm is trivial.}, such that $K=\Quot(A)$,  and $k(\fp) = \Quot(A/\fp)$, $\fp \subset A$ a prime ideal, are computable fields.
Let the ideal $I\subset A[x]$ be given by a finite set $S$ of polynomials in $A[x]$.  Choose a fixed  local degree ordering $>$ on $Mon(x)$. 
For the standard basis computation in the following algorithm we use  \cite[Algorithm 1.7.1 and 1.7.6] {GP08}.
\medskip

\begin{tabular}{lll}
INPUT: & $S\subset A[x]$ a finite set of polynomials.\\ 
& Assume that  $d(0):= \dim_K K[x]_\x/I(0)< \infty$,\\  
& with $I(0)$ the ideal generated by $S$ in $K[x]_\x$. \\

OUTPUT: & $G\subset A[x]$ a standard basis w.r.t. $>$ for the ideal $I(0)$.\\
\end{tabular}
\begin{enumerate}
\item Choose a prime ideal $\fp\neq \langle 0 \rangle$ in $A$ and  compute a standard  basis $G(\fp)$ w.r.t. $>$ of the ideal $I(\fp)$ generated by $S$ in $k(\fp) [x]_\x$.
\item Use $G(\fp)$ to compute  $d(\fp):=\dim_{k(\fp)}k(\fp) [x]_\x/I(\fp)$. 
\item If $d(\fp) = \infty$ choose another $\fp\neq \langle 0 \rangle$ and continue from the beginning.
\item Assume  $d(\fp) <\infty$ and compute $HC(I(\fp))$.
\item Compute a standard basis $G$, starting with $S\subset K[x]$,  and omit any non-vanishing term of degree $ > \deg(HC(I(\fp)))+1$ during the computation. $G$ is a standard basis of an ideal $I'(0) \supset I(0)$ in $K[x]_\x$.
\item Compute $d(0):=\dim_K K[x]_\x/I'(0)$. 
\item If $d(0)=d(\fp)$ return $G$, else\\
choose another $\fp\neq \langle 0 \rangle$ and continue from the beginning.
\item Compute a standard basis $G$ of $I$ without omitting terms and return $G$.
\end{enumerate}

\begin{Theorem} \label{thm.alg}
\begin {enumerate}
\item  The algorithm is correct. 
\item If  $A$ contains infinitely many prime ideals, there exists an open dense and infinite subset $U\subset \Spec A$ such that 
the leading ideals of $I(0)$ and $I(\fp)$ coincide for $\fp \in U$.
In particular, the algorithm terminates with step 7. for a random choice of $\fp$ in steps 3. and 7. of the algorithm.
\item The algorithm terminates with step 8. of the algorithm, if  $A$ contains finitely many prime ideals.
\end{enumerate}
\end{Theorem}

\begin{proof}
(1) Let $M$ be the set of monomials of degree $= \deg(HC(I(\fp)))+2$. Omitting all monomials of degree  $> \deg(HC(I(\fp)))+1$ is the same as computing a standard basis of the ideal $I'(0)\supset I(0)$ generated by $S\cup M$ (see the proof of Corollary \ref{cor.trunc}). If $d(0)=d(\fp)$ then  $I'(0)=I(0)$ by Proposition \ref{prop.scHC} (i), showing that the  algorithm is correct.

(2)  During the standard basis computation only finitely many coefficients of polynomials are involved and we may use (for theoretical purposes) the symmetric form of the s-polynomials and the normal form without division  to compute a {\em pseudo standard basis} of $I$ in the sense of \cite[Exercise 2.3.7]{GP08}. 
This has coefficients in $A$ and is a standard basis of the ideal $I(0)$ over the field $K$ and of $I(\fp)$ over $A/\fp$ if we specialize modulo a maximal ideal $\fp$ (c.f. \cite[Exercise 2.3.6 - 2.3.9]{GP08}).
Let $a \in A$ be the product of the leading coefficients of the elements of the pseudo standard basis (hence $a\neq0$) and set $U:= \Spec A \smallsetminus V(a)$, where $V(a)$ denotes the hypersurface in $\Spec A$ defined by $a$. 

For $\fp \in U\smallsetminus \langle 0\rangle$  the coefficients of the leading monomials of a pseudo standard basis  will not vanish mod $\fp$ and the leading ideals of $I(0)$ and $I(\fp)$ coincide. In particular, $d(\fp)= d(0)$ for $\fp \in U$, and hence the algorithm terminates for $\fp \in U\smallsetminus \langle 0\rangle $.

We show that U is infinite if $\Spec A$ is infinite. We have $U=\Spec B$,  $B=A[T]/\langle aT-1\rangle$ irreducible,   $\dim B= \dim A$ and $V(a) = \Spec A/ \langle a\rangle$ with  $\dim  A/ \langle a\rangle = \dim A -1.$ 
Assume that $\Spec U$ is finite, i.e., $B$ contains only finitely many prime ideals.
Then $\dim A = \dim B =1$ by Remark \ref{rem.infprimes} and $\dim  A/ \langle a\rangle =0$. Then $V(a)$ is finite and this implies $\Spec A = V(a) \cup U$ is finite, a contradiction.\\
Since $A$ is irreducible $U$ is open and dense in $\Spec A$ and a random choice of $\fp$ will hence pick $\fp \in U\smallsetminus \langle 0\rangle$.\\
(3) The algorithm terminates if $A$ has only finitely many prime ideals, since it will be decided after finitely many steps if the equalities in step 3. and 7 hold. If this is not the case, the algorithm stops with step 8.
 \end{proof}

\begin{Remark} \label{rem.infprimes}
{\em
(1) Every Noetherian ring $A$ of dimension $\geq 2$ contains infinitely many 
prime ideals. This follows from \cite[Theorem 31.2]{Ma86}.\\
(2) By Euclid, $\Z$ contains infinitely many prime ideals. The same argument shows that every polynomial ring $A=k[t_1,...,t_s]$, $s\geq 1$, contains infinitely many irreducible polynomials and hence infinitely many prime ideals (since  $A$ is a UFD every irreducible polynomial generates a prime ideal).\\
(3)  On the other hand, any local one-dimensional domain has only two  prime ideals. \\
(4) If $A$ has only finitely many prime ideals, the algorithm may terminate with step 7. if the equalities in 3. and 7. hold for some $\fp \neq 0$. 
But this may not happen and we added the step 8. only for completeness (without anything new).
}
\end{Remark}

\begin{Corollary}
Let $A$ be a principal ideal domain with infinitely many prime ideals. Then any non-zero prime ideal $\fp$ is maximal with $k(\fp) = A/\fp$ and the set $\{\fp \in \Spec A| I(0) \neq I(\fp) \}$ is finite.
Hence the algorithm terminates with step 7. for any (not necessarily random) choice of $\fp$  in steps 3. and 7. of the algorithm.
\end{Corollary}

\begin{proof} Any non-zero prime $\fp$ is a maximal ideal  since the Krull dimension of $A$ is 1. Since $A$ is factorial, the element $a\in A$  in the proof of Theorem \ref{thm.alg} is a product of finitely many irreducible factors defining the prime ideals  $\fp$ such that $I(0) \neq I(\fp)$.
\end{proof}

For example, the corollary applies to  $A=\Z$ and to $A=k[t]$, $t$ one variable, $k$ any field. For $k$ infinite there are even infinitely many maximal ideals of the form $\langle t-p \rangle$, $p \in k$. 

In practise $d(0)=d(\fp)$ will be usually the case for the first choice of a random $\fp$.

\begin{Remark}{\em 
For computations of standard bases over $\Q$ modular algorithms are often much faster than a direct computation. One
computes several standard bases for different prime numbers, lifts them to $\Q$ and gets a potential standard basis over $\Q$ (cf. \cite{Ar03} for global and \cite{Pf07} for local orderings).
To prove correctness a standard basis computation of the lifted basis has to be performed over $\Q$, which is usually the most time consuming part. We mention that for 0-dimensional ideals and local degree orderings this last computation can be speeded up by using the highest corner from the modular computations. 

However, the examples below show that Algorithm \ref{alg} is often faster than the modular algorithm (without using the highest corner).
}
\end{Remark}

\section {Examples}
The following examples demonstrate the effect of Algorithm \ref{alg}. We show the  {\sc Singular} code for the examples, list then the examples of ideals to be computed, and finally give a table of timings.  
\begin{verbatim}
ring R  = 0,(x,y,z),ds;	       
poly F  = x3y3+x5y2+2x2y5+x2y2z3+xy7+z9+y13+x25;
ideal I = jacob(F),F;           
ring r  = 320039,(x,y,z),ds;
ideal I = imap(R,I);       //maps I from R to r
ideal J = std(I);
poly HC = highcorner(J);   //x24z7
setring R;
noether = z*imap(r,HC);
ideal J = std(I);
noether = 0;
ideal J = std(I);
\end{verbatim}
Comment on the {\sc Singular} code:\\
{\tt ring R  = 0,(x,y,z),ds;}: the ring $\Q[x,y,z]_{\langle x,y,z\rangle}$  with  local 
degree reverse lexicographical ordering $ds$, \\ 
 {\tt ideal I = jacob(F),F;}: the ideal generated by  {\tt F} and the partials of {\tt F}, \\
{\tt  ideal J = std(I);}:  a standard basis in $\F_p[x,y,z]_{\langle x,y,z\rangle}$, $p=320039$,\\
 {\tt poly HC = highcorner(J);}: the highest corner of $J$, $x^{24}z^7$in this case,\\
 {\tt noether = z*imap(r,HC);}: the {\sc Singular} command for truncating terms bigger (w.r.t.  {\tt ds}) than $z*HC$.\\
 {\tt  ideal J = std(I);}: a standard basis in $\Q[x,y,z]_{\langle x,y,z\rangle}$ using $HC$ for truncation,\\
{\tt noether = 0;}: disables truncation during standard basis computation,\\
{\tt ideal J = std(I);}: a standard basis in $\Q[x,y,z]_{\langle x,y,z\rangle}$ not using $HC$.\\

The first 4 examples are computed over $\Q$ and $\mathbb F_p$. The {\sc Singular} code is the same for the first 3 examples, with monomial ordering $ds$ and  computations in $\F_{320039}[y,x,z]_{\langle x,y,z\rangle}$ resp. in $\Q[y,x,z]_{\langle x,y,z\rangle}$.

 \begin{enumerate}
 \item $F  = x^3y^3+x^5y^2+2x^2y^5+x^2y^2z^3+xy^7+z^9+y^{13}+x^{25},$\\
$I = \langle F, \frac{\partial F}{\partial x},\frac{\partial F}{\partial y},\frac{\partial F}{\partial z}\rangle$.   

\item $F = xyz(x+y+z)^2 +(x+y+z)^3 +x^{15}+y^{15}+z^{15}$,\\
$ I =  \langle \frac{\partial F}{\partial x},\frac{\partial F}{\partial y},\frac{\partial F}{\partial z}\rangle$.

\item $F= x^8y^6+x^{10}y^5+x^8y^7+2x^7y^8+x^7y^6z^2+x^{16}+x^6y^{10}+y^{18}+z^{20},$\\
$ I =  \langle \frac{\partial F}{\partial x},\frac{\partial F}{\partial y},\frac{\partial F}{\partial z}\rangle$.

\item 
 $I =$ ideal in the ring $\Z[x,y,z,w]_{\langle x,y,z,w\rangle}$ having 5 generators, which are random linear combinations of polynomials of degree 5, 7 and 10, with integer coefficients in the interval [-99,99]. The
{\sc Singular} code for creating this:

\begin{verbatim}
LIB "random.lib";
ring R = 0,(x,y,z,w),ds;
system("random",100);    //sets the start value of random
ideal I=randomid(maxideal(5)+maxideal(7)+maxideal(10),5,99);
\end{verbatim}
 
 \text{    } In the following 4 examples the ring has a parameter $t$ and we want to compute the ideals in $\Q(t)[x,y,z]_{\langle x,y,z\rangle}$ resp. in $\Q(t)[x,y,z,w]_{\langle x,y,z,w\rangle}$. We set the parameter to $t=1$ to compute the  highest corner in $\mathbb F_{32003} [y,x,z]_{\langle x,y,z\rangle}$ resp. in $\mathbb F_{32003}[x,y,z,w]_{\langle x,y,z,w\rangle}$.
 
\begin{verbatim} 
ring R = (0,t),(x,y,z),ds;
poly F = y10+(t2)*x7y7+x15+x9y6+(2t)*x6y9+x6y6z3+x5y11+z21;
poly F1 = subst(F,t,1);      //setting t=1
ring r = 32003,(x,y,z),ds;
poly F = imap(R,F1);
ideal I = jacob(F);
ideal J = std(I);
poly HC = highcorner(J);     //x7y2z37
setring R;
noether = z*imap(r,HC); 
ideal I = jacob(F);
ideal J = std(I);
noether = 0;
ideal J = std(I);
\end{verbatim} 

\item
$F = y^{10}+t^2x^7y^7+x^{15}+x^9y^6+2tx^6y^9+x^6y^6z^3+x^5y^{11}+z^{21},$\\
$ I =  \langle \frac{\partial F}{\partial x},\frac{\partial F}{\partial y},\frac{\partial F}{\partial z}\rangle$.
\item
$F = xyz(x+y+z)^2 +(x+y+z)^3 + t(x^{15}+y^{15}+z^{15})$,\\
$ I =  \langle \frac{\partial F}{\partial x},\frac{\partial F}{\partial y},\frac{\partial F}{\partial z}\rangle$.
\item
$F = x^8y^6+x^{10}y^5+x^8y^7+2x^7y^8+x^7y^6z^2+x^{16}+x^6y^{10}+ty^{18}+t^2z^{20}$,\\
$ I =  \langle \frac{\partial F}{\partial x},\frac{\partial F}{\partial y},\frac{\partial F}{\partial z}\rangle$.

\item  
$I =$ ideal in the ring $\Z(t)[x,y,z,w]_{\langle x,y,z,w\rangle}$ having 5 generators, which are random linear combinations of polynomials of degree 5, 7 and 9, with integer coefficients in the interval [-99,99]. 
The {\sc Singular} code for creating this:

\begin{verbatim}  
LIB "random.lib";
ring R = (0,t),(x,y,z,w),ds;
system("random",100); //sets the start value of random 
ideal I=randomid(maxideal(5)+maxideal(7)+t*maxideal(9),5,99);}
\end{verbatim}
\end{enumerate}

\noindent{\large {\bf Timings} }\\
The examples were computed on a Linux machine with i7-6700 CPU @ 3.40GHz. The times in the columns "Alg. \ref{alg}" (referring to our algorithm) and "std without HC" (referring to the usual standard basis algorithm), and "modStd" are in seconds. RAM is the maximum memory requirement (for the usual algorithm) in MB. \
{\tt modStd} is the modular standard basis algorithm in {\sc Singular}, it is not implemented for rings with parameters.
\begin{center}
\[
\begin{array}{| c | r | r | r | r |} \hline
Ex. & \text{Alg. \ref{alg}} &  \text{{\tt std} without HC} & \text{  RAM }&{\tt modStd }\\ \hline
1 & 0.03 & 2115.29 & 7571&2.60\\
2 & 0.16 &210.69  & 1213 &5.95\\ 
3 & 3.36 &> 2h & > 64 GB &> 2h\\
4 & 124.94 &2457.66 & 7938&141.74\\ \hline
5 & 0.01 &> 2h & > 120GB&--\\ 
6 & 0.16 &> 2h & > 80GB	&--\\ 
7 & 13.59 &> 2h& > 80GB&--\\ 
8 & 50.32 &1881.670 & 10003 &--\\ \hline
\end{array}
\]
\end{center}

%%%%%%%%%%%%%%%%%%%%%%%%%%%%%%%%%%%%%%%%%%%

\end{document}